\documentclass[a4paper,oneside,10.5pt]{article}%
\usepackage{amsmath}
\usepackage{amsfonts}
\usepackage{indentfirst}
\usepackage{amssymb}
\usepackage{graphicx}%
\usepackage{appendix}

\usepackage[colorlinks,citecolor=blue,urlcolor=blue]{hyperref}

\providecommand{\U}[1]{\protect \rule{.1in}{.1in}}

\newtheorem{theorem}{Theorem}[section]

\newtheorem{corollary}[theorem]{Corollary}

\newtheorem{lemma}[theorem]{Lemma}

\newtheorem{proposition}[theorem]{Proposition}
\newtheorem{remark}[theorem]{Remark}

\newenvironment{proof}[1][Proof]{\noindent \textbf{#1.} }{\  \rule{0.5em}{0.5em}}
\numberwithin{equation}{section}

\begin{document}
	
	\title{Stochastic Maximum Principle for Optimal Control of Anticipated Backward Stochastic Systems with Delays}
	\author{Guanwei Cheng\thanks{School of Mathematics, Shandong University, PR China, (guanwei.cheng@mail.sdu.edu.cn).}
		\thanks{This work was supported by National Natural Science Foundation of China (Grant No.12471450), and Taishan Scholar Talent Project Youth Project.}}
	\date{December 2025}
	\maketitle
	
	\textbf{Abstract}:  This paper investigates optimal control problems for delayed systems governed by Infinitely Anticipated Backward Stochastic Differential Equations (IABSDEs). Unlike existing frameworks limited to bounded delays, we introduce a generalized formulation utilizing $\sigma$-finite measures that accommodates both long-term memory effects and forward-looking anticipation. 
	Employing a new type of infinitely delayed stochastic differential equations as adjoint equations, we derive the necessary conditions of the maximum principle for optimal control. Under appropriate assumptions, the sufficiency of the maximum principle is also established. As illustrative examples, a climate policy model, a consumption optimization problem and a linear quadratic control problem are discussed, and all optimal controls are derived explicitly.
		
		\addcontentsline{toc}{section}{\hspace*{1.8em}Abstract}
		
		\section{Introduction}
		
		
		Time-delay control systems have become fundamental mathematical frameworks for modeling phenomena where current states depend not only on present conditions but also on historical trajectories (see, e.g., Mohammed \cite{Mohammedbook1984}). These systems naturally arise in various real-world fields, including economics, finance, engineering, and biological processes (see \cite{Mohammedarticle2007, Mohammedbook1996}).
		
		
		
		In some cases, delayed optimization problems can be solved explicitly using the Dynamic Programming Principle (DPP) approach (see \cite{solvablecase}). However, DPP faces significant challenges when applied to general delay systems due to the infinite-dimensional nature of the associated Hamilton-Jacobi-Bellman equation, which is notoriously difficult to solve (see \cite{Gozzi2009, Gozzi2017}).
		
		
		
		An alternative approach, the Stochastic Maximum Principle (SMP), has been proposed to address optimal control problems for stochastic delay systems. Oksendal and Sulem \cite{Oksendal2} presented SMP by utilizing the Anticipated Backward Stochastic Differential Equation (ABSDE) as the	adjoint equation in duality with stochastic differential delay equations (SDDE). This duality was further explored in Peng and Yang \cite{peng09AP}, forming the theoretical foundation for deriving the SMP in systems described by SDDEs. 
		For example, Wu and Chen \cite{SDDE1wuzhenchenli} applied this duality to solve optimal control problems with point-wise delays, while Yu \cite{YU20122420} considered delayed systems with impulse control. Subsequent research has expanded these studies in various directions, see e.g. \cite{Dahl03072020, wangmeng25, SDDE123mixeddelay} and the references therein.

		From a modeling perspective, control problems for BSDEs have their own research interest. Delong and Imkeller \cite{Delong1} introduced BSDEs with time-delayed generators (BSDDE), while Chen and Huang \cite{BSDDE-12.delay-1type-1} investigated related delay systems with applications to insurance risk management, pricing, and hedging (see also Delong \cite{Delong2, Delong3}). Notably, ABSDE-driven optimization problems have been widely studied for their inherent after-effect phenomena, as observed in various research (e.g., \cite{FD-BA-1delay, FD-BA-1delay-impulse}).
		
		
		Despite these advances, existing studies mainly focus on systems with point-wise delays or anticipation structures.
		
		
		The mathematical literature distinguishes two primary types of delay formulations: discrete delays, which capture fixed time lags, and moving-average delays, which integrate historical data over a period (see, e.g., \cite{BSDDE12delay}). The latter formulation is particularly useful in financial modeling \cite{Delong3} and control engineering \cite{SDDE-12delay}. Chen and Huang \cite{BSDDE-12.delay-1type-1} employed integral operators, such as $\int_{t-\delta}^t \phi(t, s) y(s) \alpha(d s)$, to unify point-wise and moving-average delays (or anticipations). There have been fruitful results on SMP for delayed control systems with such integral (see e.g. \cite{BSDDE-12delay-type1+IH, wangmeng25}). However, existing studies generally require the delay (or anticipation) window $\delta$ be sufficiently small or
		confined to fixed intervals to ensure well-posedness, thus limiting the applicability of these methods to systems exhibiting long-term memory effects.

		Motivated by these theoretical and applied considerations, this paper aims to investigate the SMP for a novel class of delay systems characterized by Infinitely Anticipated Backward Stochastic Differential Equations (IABSDEs):
		
		
		\begin{equation} \label{ABSDE with control}
			\left\{\begin{aligned}
				-dY(t)=&f\left(t, Y(t), E^{\mathcal{F}_t}\left[\int_{0}^{+\infty}\phi(t,t+\theta)Y(t+\theta)\alpha(d\theta)\right], \right.\\ 
				&\quad\quad \left.   Z(t),E^{\mathcal{F}_t}\left[\int_{0}^{+\infty}\phi(t,t+\theta)Z(t+\theta)\alpha(d\theta)\right],\right.\\
				&\quad\quad \left. 
				v(t), \int_{0}^{+\infty}\phi(t-\theta,t)v(t-\theta)\alpha(d\theta)\right)
				d t\\
				&\quad- Z(t) d W(t), \quad\quad\quad\quad\quad\quad\quad\quad\quad\quad\quad\quad\quad~  t \in[0, T]; \\
				Y(t) =  \xi&(t), Z(t) = \eta(t), \quad t\in[T, +\infty),\quad  v(t)=\varphi(t), \quad t\in(-\infty, 0].
			\end{aligned}\right.
		\end{equation}
		
		
		One important feature of system (\ref{ABSDE with control}) is that both the anticipated state terms and the delayed control term are presented in the integral form that contains a bounded density $\phi$ and $\sigma$-finite measure $\alpha$ on $[0,+\infty)$, enabling full information of unbounded past and future.
		
		
		Another interesting character emerges when we show that the adjoint equation associated with system (\ref{ABSDE with control}) formulates an Infinitely Delayed Stochastic Differential Equation (ISDDE), a class of memory processes with diverse physical and economic applications such as species growth and carbon release model (see \cite{ISDDEapplication1, ISDDEapplication2}).
		

		
		The main contribution of this paper are threefold: 
		\begin{itemize}
			\item \textbf{Well-posedness of IABSDEs}: We establish well-posedness of anticipated BSDEs with unbounded anticipation in integral form, overcoming the restrictions related to the size of the information window $\delta$ in prior work (see e.g., \cite{BSDDE-12.delay-1type-1, wangmeng25}). This enables modeling of long-term memory and aftereffects commonly encountered in economic and biological systems.
			
			
			\item \textbf{New duality for SMP}: We employ a novel duality between ISDDEs and IABSDEs to derive the SMP. This duality allows the system (\ref{ABSDE with control}) to account for both long historical dependence and anticipation simultaneously, which previously requires separate analysis.
			
			
			\item \textbf{Sufficiency of the maximum principle and explicit solvability}: We prove the sufficiency of the maximum principle under appropriate convexity assumptions. The theoretical results are further illustrated through several examples, for which the optimal controls admit explicit representations.

		\end{itemize}
		

		
		The remainder of this paper is organized as follows. Section \ref{chapter preliminary} presents preliminaries and some properties of Infinitely Anticipated BSDEs and Infinitely Delayed SDEs. In Section \ref{section 3}, we establish the necessary and sufficient conditions for the SMP. 
		Finally, Section \ref{chapter applications} 
		presents several application examples illustrating how the proposed
		framework covers different types of anticipation and delay mechanisms commonly encountered in practice.
		
		
		\section{Preliminary}\label{chapter preliminary}
		
		\indent
		Let $\left(\Omega, \mathcal{F}, P, \mathcal{F}_t, t \geqslant 0\right)$ be a complete filtered probability space such that $\mathcal{F}_0$ contains all $P$-null sets in $\mathcal{F}$, and suppose that the filtration $\mathcal{F}_t$ is generated by a $d$-dimensional standard Brownian motion $W=\left(W_t\right)_{t \geqslant 0}$. Let $T>0$ be the finite time horizon. For all $n \in \mathbb{N}$, denote the Euclidean norm on $\mathbb{R}^d$ by $|\cdot|$. 
		
		We first define some spaces in finite horizon: 
		\begin{itemize}
			\item $L^2\left(\mathcal{F}_t ; \mathbb{R}^d\right)=\Big\{\mathbb{R}^d$-valued $\mathcal{F}_t$-measurable random variables such that $E\left[|\xi|^2\right]<\infty\Big\}$;
			\item $\mathcal{M}_{\mathcal{F}}^2\left(0, T ; \mathbb{R}^d\right)=\Big\{\mathbb{R}^d$-valued and $\mathcal{F}_t$-adapted stochastic processes such that\\ $E\left[\int_0^T\left|\varphi_t\right|^2 d t\right]<\infty\Big\} ;$
			\item $\mathcal{S}_{\mathcal{F}}^2\left(0, T; \mathbb{R}^d\right)=\Big\{$continuous stochastic processes in $\mathcal{M}_{\mathcal{F}}^2\left(0, T ; \mathbb{R}^d\right)$ such that\\ $E\left[\sup _{0 \leqslant t \leqslant T}\left|\varphi_t\right|^2\right]<\infty\Big\}.$
		\end{itemize}
		
		For the stochastic processes in infinite horizon, we denote, for some given constant $\beta$: 
		\begin{itemize}
			\item $\mathcal{M}_{\mathcal{F}}^2\left(-\infty, T ; \mathbb{R}^d\right)=\Big\{\mathbb{R}^d$-valued and $\mathcal{F}_t$-adapted stochastic processes such that\\ $E\left[\int_{-\infty}^T\left|\varphi_t\right|^2 d t\right]<\infty\Big\} ;$
			\item $\mathcal{M}_{\mathcal{F}}^{2,\beta}\left(0, +\infty ; \mathbb{R}^d\right)=\Big\{\mathbb{R}^d$-valued and $\mathcal{F}_t$-adapted stochastic processes such that\\ $E\left[\int_0^{+\infty}e^{\beta t}\left|\varphi_t\right|^2 d t\right]<\infty\Big\} ;$
			\item $\mathcal{S}_{\mathcal{F}}^2\left(0, +\infty ; \mathbb{R}^d\right)=\Big\{$continuous stochastic processes in $\mathcal{M}_{\mathcal{F}}^2\left(0, +\infty ; \mathbb{R}^d\right)$ such that\\ $E\left[\sup _{0 \leqslant t \textless +\infty}\left|\varphi_t\right|^2\right]<\infty \Big\}.$
		\end{itemize}
		

		Consider the following Infinitely Anticipated BSDEs:
		\begin{equation} \label{ABSDE without control}
			\left\{\begin{aligned}
				-dY(t)=f&\left(t, Y(t), E^{\mathcal{F}_t}\left[\int_{0}^{+\infty}\phi(t,t+\theta)Y(t+\theta)\alpha(d\theta)\right], \right.\\ 
				&\quad~~ \left.  Z(t), E^{\mathcal{F}_t}\left[\int_{0}^{+\infty}\phi(t,t+\theta)Z(t+\theta)\alpha(d\theta)\right]\right)
				d t\\
				&- Z(t) d W(t), \quad\quad\quad\quad\quad\quad\quad\quad\quad\quad  t \in[0, T]; \\
				Y(t) =  \xi(t&), \quad Z(t) = \eta(t), \quad \quad\quad\quad\quad\quad\quad\quad\quad t\in[T, +\infty).
			\end{aligned}\right.
		\end{equation}
		where $\phi(\cdot, \cdot)$ is a bounded process and $\alpha$ is a $\sigma$-finite measure on $[0,+\infty)$. Denote by $\left|\alpha\right|$ the total variation of measure $\alpha$ and assume:
		
		\textbf{(H2.1)}  There exists constants $C_{\phi}>0$ and $C_{\alpha}>0$ such that 
		$$
		|\phi(\cdot,\cdot)| \leqslant C_{\phi}, 
		\quad\int_{0}^{+\infty}\left|\alpha\right|(d\theta) \leqslant C_{\alpha}.
		$$

		%
		We further assume for all $t\in[0,T]$, the generator $f: \Omega \times \mathbb{R}^d \times \mathbb{R}^d \times  \mathbb{R}^{m \times d}\times \mathbb{R}^{m \times d}$ $\rightarrow  \mathbb{R}^d$ is a $\mathcal{F}$-adapted function satisfying the following conditions: 
		
		\textbf{(H2.2)} There exists a constant $L>0$, such that for all $t \in[0, T], y, y^{\prime}, y_a, y_a^{\prime} \in \mathbb{R}^d, z, z^{\prime}, z_a, z_a^{\prime}\in \mathbb{R}^{m \times d}$, it follows that 
		$$
		\begin{aligned}
			\left|f\left(t, y,y_a, z, z_a\right)-f\left(t, y^{\prime}, y_a^{\prime}, z^{\prime},  z_a^{\prime}\right)\right| 
			\leqslant L\left(\left|y-y^{\prime}\right| + \left|y_a-y_a^{\prime}\right| +\left|z-z^{\prime}\right| + \left|z_a-z_a^{\prime}\right|\right).
		\end{aligned}
		$$
		
		\textbf{(H2.3)}  
		$$E\left[\int_0^T|f(t, 0,0,0,0)|^2 d t\right]<+\infty$$
		
		\begin{remark}
			The anticipation operator such as
			\begin{equation}\label{anticipation term}
				Y_{a}(t):=E^{\mathcal{F}_t}\left[\int_{0}^{+\infty}\phi(t,t+\theta)Y(t+\theta)\alpha(d\theta)\right]
			\end{equation}
			encompasses two principal types of anticipatory structures in stochastic systems:
			\begin{itemize}
				\item \textbf{Point-wise anticipation:} Characterized by $$E^{\mathcal{F}_{t}}[Y(t+\delta)],$$ 
				the point-wise type of anticipation was first introduced in \cite{peng09AP} and represents a fundamental type of anticipatory dependence, as explored in \cite{SDDE1wuzhenchenli, FD-BA-1delay}. 
				Such structure can be embedded into the general anticipation operator (\ref{anticipation term}) by specifying the measure $\alpha$ as a Dirac measure with lag $\delta$. Detailed applications of point-wise anticipation are discussed further in Section \ref{section example 2}.
				
				\item \textbf{Moving-average anticipation:} 
				\subitem
				\textbf{(i) Exponentially weighted}: Exponentially weighted structures are commonly employed in financial modeling, especially for optimizing discounted utility, as illustrated in \cite{SDDE-12delay, Oksendal2}. 
				The associated adjoint equation typically contains terms like
				$$
				E^{\mathcal{F}_t}\left[\int_{t}^{t+\delta}e^{-\lambda(s-t)}Y(s)ds\right],
				$$
				which is a special case of (\ref{anticipation term}) by setting $\alpha$ to be absolutely continuous with respect to the Lebesgue measure on $[0,\delta]$ with exponential density $\phi$.  We will further propose a climate policy example with delay and anticipation of such type in Section \ref{section example carbon}.
				
				\subitem
				\textbf{(ii) Uniformly weighted}: Another form of moving-average anticipation involves uniformly weighted anticipation terms (see Section 4.1 in \cite{BSDDE-12.delay-1type-1}):
				$$
				E^{\mathcal{F}_t}\left[\int_{t}^{T}Y(s)ds\right],
				$$
				which can be reformulated as
				$$
				E^{\mathcal{F}_t}\left[\int_{0}^T T Y(t+\theta) I_{[0 ,T-t]}(\theta) \alpha(d \theta)\right],
				$$
				where $\alpha$ represents a uniform measure on $[0,T]$ and $I$ denotes the indicator function. Section \ref{section example 1} provides a detailed application of this approach in a consumption optimization problem.
			\end{itemize}

		\end{remark}

		
		Now we study the well-posedness of the IABSDE (\ref{ABSDE without control}). To analyze this equation, it is convenient to place it within a broader class of anticipated backward stochastic systems of the form:
		
		
		\begin{equation} \label{general ABSDE}
			\left\{\begin{aligned}
				-dY(t)= f&\left(t, \{Y(r)\}_{r \in [t,+\infty)}, \{Z(r)\}_{r \in [t,+\infty)} \right) d t-  Z(t) d W_t, \quad &&t \in[0, T]; \\
				Y(t) =  \xi(t)&, \quad Z(t) =  \eta(t),\quad && t\in[T, +\infty).
			\end{aligned}\right.
		\end{equation}
		
		Equation (\ref{general ABSDE}) was introduced in Cheng and Yang \cite{ziji}, and its existence and uniqueness have been established under a suitable Lipschitz condition as follows:
		
		
		\textbf{(H2.4)} There exists a constant $L>0$ such that for all $t \in[0, T], Y(\cdot), Y^{\prime}(\cdot) \in \mathcal{M}_\mathcal{F}^2\left(t, +\infty; \mathbb{R}^d\right)$, $Z(\cdot), Z^{\prime}(\cdot) \in \mathcal{M}_\mathcal{F}^{2,\beta}\left(t, +\infty; \mathbb{R}^{d \times m}\right)$, and an arbitrary constant $\beta \geqslant 0$, it follows that
		\begin{equation} \nonumber    
			\begin{aligned}
				&E\left[\int_{t}^{T}\left|f(s,Y(\cdot),Z(\cdot))-f(s,Y^{\prime}(\cdot),Z^{\prime}(\cdot)) \right|^2 e^{\beta s} ds\right]\\
				&\quad \leqslant LE\left[\int_{t}^{+\infty}\left(\sup_{s \leqslant r \textless +\infty}\left|Y(r) - Y^{\prime}(r)\right|^2 + \left|Z(s) - Z^{\prime}(s)\right|^2\right)e^{\beta s} ds\right].
			\end{aligned}
		\end{equation} 
		
		The key observation is that the IABSDE (\ref{ABSDE without control}) can be embedded into the more general setting of (\ref{general ABSDE}) by the following lemma.
		
		
		\begin{lemma}\label{lemma connection}
			Under assumptions (H2.1) and (H2.2), the generator of the IABSDE (\ref{ABSDE without control}) satisfies the Lipschitz condition (H2.4) required for equation (\ref{general ABSDE}).
			
		\end{lemma}
		\begin{proof}
			Since
			$$
			\begin{aligned}
				\left|Y_a(t)-Y_a^{\prime}(t)\right|&\leqslant E^{\mathcal{F}_t}\left[\int_{t}^{+\infty}\left|\phi(t,t+\theta)\right|\cdot\left|Y(t+\theta)-Y^{\prime}(t+\theta)\right|\left|\alpha\right|(d\theta)\right]\\
				&\leqslant C_{\alpha}C_{\phi}\cdot E^{\mathcal{F}_t}\left[\sup_{t \leqslant s \textless +\infty}\left|Y(s)-Y^{\prime}(s)\right|\right].
			\end{aligned}   	 
			$$ 	  
			Thus condition (H2.4) holds for anticipation term of $Y$. And for $Z$ anticipation term, for any constant $\beta \geqslant 0$,
			by Fubini's Theorem,
			$$
			\begin{aligned}
				E&\left[\int_{t}^{T} \left|Z_a(s)-Z_a^{\prime}(s)\right|^2 e^{\beta s} ds\right]\\
				&\leqslant C_{\alpha}C_{\phi}^2\cdot E\left[ \int_{t}^{T} \left(\int_{0}^{+\infty}\left|Z(s+\theta) - Z^{\prime}(s+\theta)\right|^2 \left|\alpha\right|(d\theta)\right)e^{\beta s}ds \right]\\
				&=C_{\alpha}C_{\phi}^2\cdot E\left[\int_{0}^{+\infty} \left(\int_{t+\theta}^{T+\theta} \left|Z(s) - Z^{\prime}(s)\right|^2e^{\beta s}ds\right) e^{-\beta \theta}\left|\alpha\right|(d\theta) \right]\\
				&\leqslant C_{\alpha}^2C_{\phi}^2\cdot E\left[\int_{t}^{+\infty} \left|Z(s) - Z^{\prime}(s)\right|^2e^{\beta s}ds \right].
			\end{aligned}
			$$   
		\end{proof}	     

		With Lemma \ref{lemma connection}, we therefore establish the well-posedness of equation (\ref{ABSDE without control}) directly from Theorem 4.1 in Cheng and Yang \cite{ziji}.
		

		\begin{theorem}\label{well-posedness IABSDE}
			Suppose (H2.1)-(H2.3) hold, then for any given terminal condition $\xi(\cdot)\in\mathcal{S}_\mathcal{F}^2\left(T, +\infty; \mathbb{R}^d\right)$, $\eta(\cdot)\in\mathcal{M}_{\mathcal{F}}^{2,\beta}\left(T, +\infty; \mathbb{R}^{d \times m}\right)$, the IABSDE (\ref{ABSDE without control}) has a unique solution $\left(Y(\cdot), Z(\cdot)\right) \in \mathcal{S}_\mathcal{F}^2\left(0, +\infty; \mathbb{R}^d\right) \times \mathcal{M}_\mathcal{F}^{2,\beta}\left(0, +\infty; \mathbb{R}^{d \times m}\right)$. Moreover, the solution satisfies the priori estimate:
			\begin{equation} \label{basic estimate of IABSDE}
				\begin{aligned}
					E&\left[\sup _{0 \leqslant t \leqslant T}\left|Y(t)\right|^2+\int_0^{T}\left|Z(t)\right|^2 d t\right]\\
					&\leqslant C E\left[\sup_{T \leqslant t \textless +\infty}\left|\xi(t)\right|^2+\int_{T}^{+\infty}e^{\beta t}\left|\eta(t)\right|^2 dt+\left(\int_0^T|f(t,0,0,0,0)| d t \right)^2\right],
				\end{aligned}
			\end{equation}
			where the constant $C$ depends only on $C_{\phi}, C_{\alpha}$ in (H2.1), $L$ in (H2.2) and $T$.
		\end{theorem}
		
		We next present some results of Infinitely Delayed SDEs. This class of equations plays a key role in Section \ref{section 3}, where the adjoint process associated with the IABSDE-driven control problem is shown to satisfy such an equation.
		To this end, consider the ISDDE of the form
		\begin{equation} \label{ISDDE without control}
			\left\{\begin{aligned}
				dX(t)=&b\left(t, X(t), \int_{0}^{+\infty}\phi(t-
				\theta,t)X(t-\theta)\alpha(d\theta)\right)d t \\ 
				&+\sigma\left(t, X(t), \int_{0}^{+\infty}\phi(t-
				\theta,t)X(t-\theta)\alpha(d\theta)\right)d W(t),
				\quad &&t \in[0, T]; \\
				X(t) =&\varphi(t), \quad  &&t\in(-\infty,0].
			\end{aligned}\right.
		\end{equation}
		
		The drift and diffusion coefficients are assumed to satisfy the standard Lipschitz and growth conditions.
		Assume for all $t\in[0,T]$, $b: \Omega \times \mathbb{R}^d \times \mathbb{R}^d \rightarrow \mathbb{R}^d$ and $\sigma: \Omega \times \mathbb{R}^d \times \mathbb{R}^d \rightarrow \mathbb{R}^d$ are $\mathcal{F}$-adapted and satisfy the following conditions: 
		
		\textbf{(H2.5)} There exists a constant $L>0$, such that for all $t \in[0, T], x, x^{\prime}, x_d, x_d^{\prime} \in \mathbb{R}^d$, it follows that 
		$$
		\begin{aligned}
			\left|b\left(t, x, x_d\right)-b\left(t, x^{\prime}, x_d^{\prime}\right)\right|
			+\left|\sigma\left(t, x, x_d\right)-\sigma\left(t, x^{\prime}, x_d^{\prime}\right)\right| \leqslant L\left(\left|x-x^{\prime}\right|+\left|x_d-x_d^{\prime}\right|\right).
		\end{aligned}
		$$
		
		\textbf{(H2.6)}  There exists a constant $M > 0$ such that
		$$
		\sup_{0 \leqslant t \leqslant T}\left(|b(t, 0, 0)|^2 \vee|\sigma(t, 0 ,0)|^2 \right)\leqslant M .
		$$
		
		The delay operator in (\ref{ISDDE without control}) shares the same measure-valued structure as that used in the anticipation terms of the IABSDE (\ref{ABSDE without control}). Lemma \ref{lemma connection} implies that the mapping
		$$
		X(\cdot) \longmapsto \int_0^{\infty} \varphi(t-\theta, t) X(t-\theta) \alpha(d \theta)
		$$
		is a bounded linear functional on $\mathcal{M}_\mathcal{F}^2\left(-\infty, T; \mathbb{R}^d\right)$, ensuring that (\ref{ISDDE without control}) fits naturally into the general framework of ISDDE introduced in Wei and Wang \cite{07JMAA}:	
		\begin{equation} \label{general ISDDE}
			\left\{\begin{aligned}
				d X(t)=b&\left(t,\{X(r)\}_{r \in (-\infty, t]}\right) d t+\sigma\left(t,\{X(r)\}_{r \in (-\infty, t]}\right) d W(t), \quad &&t \in[0, T]; \\
				X(t)=\varphi&(t), \quad &&t\in(-\infty, 0].
			\end{aligned}\right.
		\end{equation}
		
		Apply the existence and uniqueness theorem of (\ref{general ISDDE}) in Wei and Wang \cite{07JMAA} to yield the well-posedness result of equation (\ref{ISDDE without control}).
		
		
		\begin{theorem}\label{well-posedness ISDDE}
			Suppose (H2.1),(H2.5) and (H2.6) hold, then for any given initial path $\varphi(\cdot)\in\mathcal{M}_\mathcal{F}^2\left(-\infty, 0; \mathbb{R}^d\right)$, the ISDDE (\ref{ISDDE without control}) has a unique solution $X(\cdot) \in \mathcal{M}_\mathcal{F}^2\left(-\infty, T; \mathbb{R}^d\right) $.
		\end{theorem}

		\section{Optimal Control Problem}\label{section 3}
		In this section, we study the stochastic maximum principle for an anticipative stochastic delayed control system. In Section \ref{chapter necessary SMP}, the necessary conditions will be given based on the variational inequality. We also present the sufficient optimality conditions in Section \ref{chapter sufficient SMP} with additional convex conditions.
		\subsection{Formulation of Problem}
		Now, we consider an optimal control problem in which the controlled state dynamics is described by IABSDEs (\ref{ABSDE with control}).
		Given a finite horizon $T > 0$, we assume that for all $t\in[0,T]$, $f: \Omega \times \mathbb{R}^d \times \mathbb{R}^d \times  \mathbb{R}^{m \times d}\times \mathbb{R}^{m \times d} \times \mathbb{R}^k \times \mathbb{R}^k$ $\rightarrow  \mathbb{R}^d$ is a $\mathcal{F}$-adapted function, $\xi(\cdot)\in\mathcal{S}_\mathcal{F}^2\left(T, +\infty; \mathbb{R}^d\right)$, $\eta(\cdot)\in\mathcal{M}_{\mathcal{F}}^{2,\beta}\left(T, +\infty; \mathbb{R}^{d \times m}\right)$ are given terminal conditions. We introduce the control process $v(\cdot)$ taking values in a convex set $U \subset \mathbb{R}^k$, and $\varphi$, the initial path of $ v(\cdot)$, is a given deterministic continuous function from $(-\infty, 0]$ into $U$ such that $\int_{-\infty}^0 \varphi^2(s) \mathrm{ds}<+\infty$. The set of admissible controls is defined as
		$$
		\mathcal{U}_{a d}=\left\{v(\cdot):(-\infty, T] \rightarrow U \mid v(t)=\varphi(t) \text { for } t \leqslant 0, ~~ v(\cdot) \in M_{\mathcal{F}}^2\left(0, T ; \mathbb{R}^k\right) \text { for } t \in[0,T]\right\}.
		$$
		For any admissible control $v(\cdot)\in \mathcal{U}_{ad}$, the corresponding state $\left(Y^v,Z^v\right)$ is determined by the system
		\begin{equation} \label{state equation}
			\left\{\begin{aligned}
				-dY^{v}(t)=&f\left(t, Y^{v}(t), Y^{v}_a(t), Z^{v}(t), Z^{v}_a(t),
				v(t), v_d(t)\right)d t- Z^{v}(t) d W(t), \quad  t \in[0, T]; \\
				Y^{v}(t) =&  \xi(t), Z^{v}(t) = \eta(t), \quad t\in[T, +\infty),\quad \quad v(t)=\varphi(t), \quad t\in(-\infty, 0],
			\end{aligned}\right.
		\end{equation}
		where the anticipation and delay operators are given by
		$$
		\begin{aligned}
			\Gamma^{v}_{a}(t) = &E^{\mathcal{F}_t}\left[\int_{0}^{+\infty}\phi(t,t+\theta)\Gamma^{v}(t+\theta)\alpha(d\theta)\right], \quad \Gamma = Y,Z,\\
			v_d(t) =& \int_{0}^{+\infty}\phi(t-\theta,t)v(t-\theta)\alpha(d\theta).
		\end{aligned}
		$$
		Our goal is to minimize the associated performance functional 
		\begin{equation}\label{cost functional}
			\begin{aligned}
				J(v(\cdot))=E \left[\int_{0}^{T} l\left(t, Y^{v}(t), Y_a^{v}(t), Z^v(t), Z^{v}_a(t), v(t), v_d(t)\right) d t\right]+ \mathbb{E}\left[\gamma(Y^v(0))\right],
			\end{aligned}
		\end{equation}
		where for all $t\in[0,T]$, $l: \Omega \times \mathbb{R}^d   \times \mathbb{R}^d \times  \mathbb{R}^{m \times d}\times \mathbb{R}^{m \times d} \times \mathbb{R}^k \times \mathbb{R}^k$ $\rightarrow  \mathbb{R}^d$ and $\gamma: \mathbb{R}^d$ $\rightarrow \mathbb{R}$ are given measurable functions.
		\\
		\\
		To ensure differentiability of variations, we impose the following conditions:
		
		\textbf{(H3.1)} $f$ is continuously differentiable with respect to $\left(Y,Y_a,Z,Z_a,v,v_d \right)$, and the partial derivative $\left(f_y,f_{y_a},f_z,f_{z_a},f_v,f_{v_d}\right)$ 
		are uniformly bounded. 
		
		\textbf{(H3.2)} For each admissible control $v(\cdot)$, $l\left(\cdot, Y^v(\cdot), Y_a^v(\cdot), Z^v(\cdot), Z_a^v(\cdot), v(\cdot), v_d(\cdot)\right)$ $\in \mathcal{M}_{\mathcal{F}}^{2}(0, T ; \mathbb{R})$, $l$ is differentiable with respect to $\left(Y,Y_a,Z,Z_a,v,v_d \right)$, $\gamma$ is differentiable with respect to $Y$, and all the derivatives are uniformly bounded. 
		
		\begin{remark}
			Suppose (H3.1) holds, then by Theorem \ref{well-posedness IABSDE}, for each admissible control $v(\cdot)\in \mathcal{U}_{ad}$, the state equation (\ref{state equation}) admits a unique solution $\left(Y^v(\cdot), Z^v(\cdot)\right) \in \mathcal{S}_\mathcal{F}^2\left(0, +\infty; \mathbb{R}^d\right) \times \mathcal{M}_\mathcal{F}^{2,\beta}\left(0, +\infty; \mathbb{R}^{d \times m}\right)$ corresponding to control $v$.
		\end{remark}
		
		
		Let $u(\cdot)\in \mathcal{U}_{ad}$ be the optimal control of the delayed
		stochastic control problem (\ref{cost functional}) described by anticipated BSDE (\ref{state equation}), and $Y^u(\cdot),Z^u(\cdot)$ is the corresponding optimal  trajectories. In the remainder of this chapter, we shall derive necessary and sufficient conditions for the maximum principle of the above problem.
		
		\subsection{Necessary Conditions.}\label{chapter necessary SMP}
		
		Let $v(\cdot)$ be arbitrary admissible control in $\mathcal{U}_{ad}$, for each $0\leqslant \varepsilon \leqslant 1$, take the convex perturbation $u^{\varepsilon}(\cdot)=u(\cdot)+\varepsilon\left(v(\cdot)-u(\cdot)\right)\in \mathcal{U}_{ad}$, and denote the corresponding trajectory by $Y^{\varepsilon}(\cdot),Z^{\varepsilon}(\cdot)$.
		
		We let
		$$\hat{v}(t)=v(t)-u(t),\quad \hat{v}_{d}(t)=v_{d}(t)-u_{d}(t),
		$$
		and denote by $\left(\hat{Y}(\cdot), \hat{Z}(\cdot)\right)$ the solution of the variational equation:
		\begin{equation} \label{variational equation}
			\left\{\begin{aligned}
				-d\hat{Y}(t)=&\bigg[f^u_{y}(t)\hat{Y}(t)+f^u_{y_{a}}(t)\hat{Y}_{a}(t)+f^u_{z}(t)\hat{Z}(t)+f^u_{z_{a}}(t)\hat{Z}_{a}(t)
				\\ 
				&~~ +f^u_{v}(t)\hat{v}(t)+f^u_{v_{d}}(t)\hat{v}_d(t)\bigg]dt 
				- \hat{Z}(t) d W(t), \quad\quad  t \in[0, T]; \\
				\hat{Y}(t) = \hat{Z}&(t) = 0, \quad t\in[T, +\infty),\quad\quad  \hat{v}(t)=0, \quad t\in(-\infty, 0].
			\end{aligned}\right.
		\end{equation}
		where for $\Phi=f,l$ and $k=y,y_a,z,z_a,v,v_d, \varnothing$, while $k=\varnothing$ means $\Phi$ does not take partial derivative on any of its component, we use the shorthand
		$$
		\begin{aligned}
			&\Phi^u_k(t)=\Phi_k\left(t, Y^u(t),Y^u_a(t),Z^u(t),Z^u_a(t),u(t),u_d(t)\right),\\
			&\Phi^{\varepsilon}_k(t)=\Phi_k\left(t, Y^{\varepsilon}(t),Y^{\varepsilon}_a(t),Z^{\varepsilon}(t),Z^{\varepsilon}_a(t),u^{\varepsilon}(t),u^{\varepsilon}_d(t)\right).
		\end{aligned}
		$$
		
		%

		Set 
		$$
		\begin{aligned}
			\tilde{Y}^{\varepsilon}(t)=\frac{Y^{\varepsilon}(t)-Y^{u}(t)}{\varepsilon}-\hat{Y}(t),\quad
			\tilde{Z}^{\varepsilon}(t)=\frac{Z^{\varepsilon}(t)-Z^{u}(t)}{\varepsilon}-\hat{Z}(t).
		\end{aligned}
		$$
		\begin{lemma}\label{variational estimate}
			Suppose (H3.1) holds, then we have 
			$$
			\lim_{\varepsilon \rightarrow 0}E\left[\sup_{0 \leqslant t \leqslant T}\left|\tilde{Y}^{\varepsilon}(t)\right|^2\right]=0, \quad \lim_{\varepsilon \rightarrow 0}E\left[\int_{0}^T\left|\tilde{Z}^{\varepsilon}(t)\right|^2\right]=0.
			$$
		\end{lemma}
		\begin{proof} By the definition of $\tilde{Y}^{\varepsilon}(t), \tilde{Z}^{\varepsilon}(t)$, we have
			$$
			\left\{\begin{aligned}
				-d\tilde{Y}^{\varepsilon}(t)=&\bigg[\frac{f^{\varepsilon}(t)-f^u(t)}{\varepsilon}-f^u_{y}(t)\hat{Y}(t)-f^u_{y_{a}}(t)\hat{Y}_{a}(t)-f^u_{z}(t)\hat{Z}(t)-f^u_{z_{a}}(t)\hat{Z}_{a}(t)
				\\ 
				&~~ -f^u_{v}(t)\hat{v}(t)-f^u_{v_{d}}(t)\hat{v}_d(t)\bigg]dt
				- \tilde{Z}^{\varepsilon}(t) d W(t), \quad  t \in[0, T]; \\
				\tilde{Y}^{\varepsilon}(t) = \tilde{Z}&^{\varepsilon}(t) = 0, \quad t\in[T, +\infty),\quad\quad\quad  \hat{v}(t)=0, \quad t\in(-\infty, 0].
			\end{aligned}\right.
			$$
			Applying the Mean Value Theorem to $f^{\varepsilon}(t)-f^u(t)$ in the variables $y,y_a, z, z_a, v, v_d$, we can derive  $\left(\tilde{Y}^{\varepsilon}(t), \tilde{Z}^{\varepsilon}(t)\right)$ satisfy the following IABSDE:
			\begin{equation} \label{linear variational equation}
				\left\{\begin{aligned}
					-d\tilde{Y}^{\varepsilon}(t)=&\left[A^{\varepsilon}(t)\tilde{Y}^{\varepsilon}(t)+B^{\varepsilon}(t)\tilde{Y}^{\varepsilon}_{a}(t)+C^{\varepsilon}(t)\tilde{Z}^{\varepsilon}(t)+D^{\varepsilon}(t)\tilde{Z}^{\varepsilon}_{a}(t)+R^{\varepsilon}(t)\right]dt
					\\ 
					&
					- \tilde{Z}^{\varepsilon}(t) d W(t), \quad\quad\quad\quad\quad\quad\quad\quad\quad\quad~~~  t \in[0, T]; \\
					\tilde{Y}^{\varepsilon}(t) = \tilde{Z}&^{\varepsilon}(t) = 0, \quad t\in[T, +\infty),\quad\quad \hat{v}(t)=0, \quad t\in(-\infty, 0].
				\end{aligned}\right.
			\end{equation}
			where
			$$
			\begin{aligned}
				A^{\varepsilon}(t):=&\int_{0}^{1}f_y\left(t,Y^u(t)+\lambda(Y^{\varepsilon}(t)-Y^u(t)),Y^u_a(t),Z^u(t),Z^u_a(t),u(t),u_d(t)\right)d\lambda,\\	
				B^{\varepsilon}(t):=&\int_{0}^{1}f_{y_{a}}\left(t,Y^u(t),Y_{a}^u(t)+\lambda(Y^{\varepsilon}_{a}(t)-Y_{a}^u(t)),Z^u(t),Z^u_a(t),u(t),u_d(t)\right)d\lambda,
			\end{aligned}
			$$
			and $C^{\varepsilon}(t),D^{\varepsilon}(t), E^{\varepsilon}(t), F^{\varepsilon}(t)$ are denoted similarly.
			The residual term $R^{\varepsilon}(t)$ is given by 
			$$
			\begin{aligned}
				R^{\varepsilon}(t):=&\left(A^{\varepsilon}(t)-f^u_y(t)\right)\hat{Y}(t) +  \left(B^{\varepsilon}(t)-f^u_{y_{a}}(t)\right)\hat{Y}_{a}(t) + \left(C^{\varepsilon}(t)-f^u_z(t)\right)\hat{Z}(t)\\
				& ~+ \left(D^{\varepsilon}(t)-f^u_{z_{a}}(t)\right)\hat{Z}_{a}(t)+ \left(E^{\varepsilon}(t)-f^u_{v}(t)\right)\hat{v}(t) +
				\left(F^{\varepsilon}(t)-f^u_{v_{d}}(t)\right)\hat{v}_{d}(t).
			\end{aligned}
			$$
			It follows that for all $t\in[0,T]$, $R^{\varepsilon}(t)\rightarrow0$ as $\varepsilon \rightarrow 0$. Then applying the priori estimate (\ref{basic estimate of IABSDE}) to the linear equation (\ref{linear variational equation}), we can obtain 
			$$	
			\begin{aligned}
				E\left[\sup _{0 \leqslant t \leqslant T}\left|\tilde{Y}^{\varepsilon}(t)\right|^2+\int_0^{T}\left|\tilde{Z}^{\varepsilon}(t)\right|^2 d t\right]
				\leqslant C\left(T, L, C_{\phi}, C_{\alpha}\right) E\left[\int_0^T|R^{\varepsilon}(t)|^2 d t\right] \rightarrow 0,
			\end{aligned}
			$$
			as $\varepsilon \rightarrow 0$.
		\end{proof}
		\\
		
		Since the anticipation operators $Y\mapsto Y_{a}, Z\mapsto Z_{a}$ are bounded linear functionals (see Lemma \ref{lemma connection}), we immediately obtain:
		\begin{remark}\label{variational estimate anticipated}
			Under the assumptions of Lemma \ref{variational estimate},
			$$
			\lim_{\varepsilon \rightarrow 0}E\left[\sup_{0 \leqslant t \leqslant T}\left|\tilde{Y}_{a}^{\varepsilon}(t)\right|^2+\int_{0}^T\left|\tilde{Z}_{a}^{\varepsilon}(t)\right|^2\right]=0.
			$$
		\end{remark} 
		
		
		
		From Lemma \ref{variational estimate} and Remark \ref{variational estimate anticipated}, we can derive the variational inequality.
		\begin{lemma}\label{variational inequality}
			Suppose (H3.1) and (H3.2) hold, then the following variational inequality holds:
			\begin{equation}\label{eq variational inequality}
				\begin{aligned}
					&E\bigg[ \int_{0}^{T}\bigg(l^u_y(t)\hat{Y}(t)+l^u_{y_{a}}(t)\hat{Y}_{a}(t) +l^u_z(t)\hat{Z}(t)+ l^u_{z_{a}}(t)\hat{Z}_{a}(t)\\
					&\quad\quad\quad\quad  +l^u_v(t)\hat{v}(t) +l^u_{v_{d}}(t)\hat{v}_{d}(t)  \bigg)dt + \gamma_y\left(Y^u(0)\right)\hat{Y}(0) \bigg]\geqslant 0.
				\end{aligned}
			\end{equation}
		\end{lemma}
		\begin{proof} ~Since $u(\cdot)$ is optimal, we have
			$$
			\begin{aligned}
				0\leqslant \frac{J\left(u^{\varepsilon}(\cdot)\right)-J\left(u(\cdot)\right)}{\varepsilon}&=\frac{1}{\varepsilon}E\left[\gamma\left(Y^{\varepsilon}(0)\right)-\gamma\left(Y^{u}(0)\right)+\int_{0}^{T}\left(l^{\varepsilon}(t)-l^{u}(t)\right)dt\right].
			\end{aligned}
			$$
			By Lemma \ref{variational estimate} and Taylor's expansion, we have 
			$$
			\begin{aligned}
				\frac{1}{\varepsilon}E\left[\gamma\left(Y^{\varepsilon}(0)\right)-\gamma\left(Y^{u}(0)\right)\right]&=\frac{1}{\varepsilon}E\left[\gamma\left(Y^{u}(0)+\varepsilon \hat{Y}(0)\right)-\gamma\left(Y^{u}(0)\right)+o(\varepsilon)\right]\\
				&=E\left[\gamma_y\left(Y^u(0)\right)\hat{Y}(0)\right]+o(1).
			\end{aligned}
			$$
			Similarly, using the same technique to deal with the $\left(y,y_a,z,z_a,v,v_d \right)$ component of $l$, together with the convergence in Lemma \ref{variational estimate} and Remark \ref{variational estimate anticipated},  we get 
			$$
			\begin{aligned}
				\frac{1}{\varepsilon}E\left[\int_{0}^{T}\left(l^{\varepsilon}(t)-l^{u}(t)\right)dt\right]
				=&\frac{1}{\varepsilon}E\bigg[\int_{0}^{T}\bigg(l\left(t, Y^{u}(t)+\varepsilon \hat{Y}(t),  Y^{u}_a(t)+\varepsilon \hat{Y}_a(t),Z^{u}(t)+\varepsilon \hat{Z}(t), \right.\\
				&\quad\quad\quad\quad\quad~ \left. Z^{u}_{a}(t)+\varepsilon \hat{Z}_{a}(t), u^{\varepsilon}(t),u_d^{\varepsilon}(t) \right)-l^{u}(t)\bigg) dt +o(\varepsilon)\bigg]\\
				=&E\bigg[ \int_{0}^{T}\bigg(l^u_y(t)\hat{Y}(t)+l^u_{y_{a}}(t)\hat{Y}_{a}(t) +l^u_z(t)\hat{Z}(t)\\
				&\quad\quad\quad~~  + l^u_{z_{a}}(t)\hat{Z}_{a}(t)+l^u_v(t)\hat{v}(t)+l^u_{v_{d}}(t)\hat{v}_{d}(t)  \bigg)dt + o(1)\bigg].
			\end{aligned}
			$$
			Sum the above two part up, then (\ref{eq variational inequality}) holds. 
		\end{proof}
		\\
		
		In order to derive the maximum principle, we introduce the dual process of the
		variational equation (\ref{variational equation}) as the following ISDDE:
		\begin{equation}\label{adjoint equation}
			\left\{\begin{aligned}
				d p(t)= & \left\{ f_y^u(t) p(t)+\int_0^{+\infty}f_{y_a}^u(t-\theta) \phi(t-\theta, t)p(t-\theta) \alpha(d\theta)\right. \\
				& \left.-l_y^u(t) -\int_{0}^{+\infty}l_{y_a}^u(t-\theta)\phi(t-\theta,t)\alpha(d\theta) \right\} d t \\
				& + \left\{ f_z^u(t) p(t)+\int_0^{+\infty}f_{z_a}^u(t-\theta) \phi(t-\theta, t)p(t-\theta) \alpha(d\theta)\right. \\
				& \left.-l_z^u(t) -\int_{0}^{+\infty}l_{z_a}^u(t-\theta)\phi(t-\theta,t)\alpha(d\theta) \right\} dW(t), \quad t\in[0,T]; \\
				p(0)= & -\gamma_y\left(Y^u(0)\right), \quad\quad p(t)=0, \quad t\in(-\infty,0).
			\end{aligned}\right.
		\end{equation}
		
		\begin{remark}
			We can see that for any given admissible control $v(\cdot)$, the adjoint equation (\ref{adjoint equation}) forms a linear ISDDE. 
			Under the assumptions (H2.1), (H2.5) and (H2.6), Theorem \ref{well-posedness ISDDE} ensures that (\ref{adjoint equation}) admits a unique solution in $\mathcal{M}_\mathcal{F}^2\left(-\infty, T; \mathbb{R}^d\right)$. 
		\end{remark}

		Define the Hamiltonian function $H: [0,	T] \times \mathbb{R}^d \times \mathbb{R}^d \times \mathbb{R}^{m \times d} \times \mathbb{R}^{m \times d}\times \mathbb{R}^k \times \mathbb{R}^k \times \mathbb{R}^d$ $\rightarrow  \mathbb{R}$: 
		$$
		H\left(t, y, y_a, z, z_a, v, v_d, p\right):=l\left(t, y, y_a, z, z_a, v, v_d\right)-\left\langle p, f\left(t, y, y_a, z, z_a, v, v_d\right)\right\rangle .
		$$
		
		Then the adjoint equation (\ref{adjoint equation}) can be rewritten into  stochastic Hamiltonian system type as:
		\begin{equation}
			\left\{\begin{aligned}
				d p(t)= & \left\{-H^u_y(t) -\int_{0}^{+\infty}\phi(t-\theta,t)H^u_{y_a}(t-\theta)\alpha(d\theta)\right\} d t \\
				& \left\{-H^u_z(t) -\int_{0}^{+\infty}\phi(t-\theta,t)H^u_{z_a}(t-\theta)\alpha(d\theta) \right\} dW(t), \quad t\in[0,T]; \\
				p(0)= & -\gamma_y\left(Y^u(0)\right), \quad\quad p(t)=0, \quad t\in(-\infty,0).
			\end{aligned}\right.
		\end{equation}
		where we denote $H^u_k(t)=H_k\left(t, Y^u(t),Y^u_a(t),Z^u(t),Z^u_a(t),u(t),u_d(t),p(t)\right)$ for $k=y,y_a,z,z_a, v,v_d$.
		\\
		
		Now we give the main result of this paper.
		\begin{theorem}\label{necessary SMP}
			(Necessary conditions of optimality) Assume (H3.1) and (H3.2) hold. Let $u(\cdot)$ be an optimal control of the stochastic optimal control problem with delay (\ref{cost functional}) described by (\ref{state equation}), Suppose $Y^u(\cdot),Z^u(\cdot)$ is the corresponding optimal trajectories and $p(\cdot)$ solves the adjoint equation (\ref{adjoint equation}). Then
			\begin{equation}\label{condition neccesary SMP}
				\begin{aligned}
					\left\langle H^u_v(t) + E^{\mathcal{F}_t}\left[ \int_{0}^{+\infty}\phi(t,t+\theta)H^u_{v_{d}}(t+\theta)\alpha(d\theta)\right],v-u(t)\right\rangle  \geqslant 0,
				\end{aligned}
			\end{equation}
			a.e., a.s. for all $v\in U$.
		\end{theorem}
		
		\begin{proof}
			Applying Itô's formula to $\left\langle p(t), \hat{Y}(t)\right\rangle$ on $[0,T]$ , we have
			$$
			\begin{aligned}
				E\left[\gamma_y\left(Y^u(0)\right)\hat{Y}(0) \right] = \Delta_{1} +\Delta_{2} +\Delta_{3},
			\end{aligned}
			$$
			where
	\begin{align*}
	\Delta_{1}:=\;&
	E\!\left[\int_{0}^{T}\left\langle p(t),
	-f^u_{y_{a}}(t)\hat{Y}_{a}(t)-f^u_{z_{a}}(t)\hat{Z}_{a}(t)
	\right\rangle dt \right] \nonumber\\
	&+E\!\left[\int_{0}^{T}\left\langle
	\int_{0}^{+\infty} f^u_{y_a}(t-\theta)\,\phi(t-\theta,t)\,p(t-\theta)\,\alpha(d\theta),
	\hat{Y}(t)\right\rangle dt \right] \nonumber\\
	&+E\!\left[\int_{0}^{T}\left\langle
	\int_{0}^{+\infty} f^u_{z_a}(t-\theta)\,\phi(t-\theta,t)\,p(t-\theta)\,\alpha(d\theta),
	\hat{Z}(t)\right\rangle dt \right], \displaybreak[2]\\[2pt]
	\Delta_{2}:=\;&
	E\!\left[\int_{0}^{T}\left\langle
	-l^u_y(t)-\int_{0}^{+\infty}\phi(t-\theta,t)\,l^u_{y_a}(t-\theta)\,\alpha(d\theta),
	\hat{Y}(t)\right\rangle dt \right] \nonumber\\
	&+E\!\left[\int_{0}^{T}\left\langle
	-l^u_z(t)-\int_{0}^{+\infty}\phi(t-\theta,t)\,l^u_{z_a}(t-\theta)\,\alpha(d\theta),
	\hat{Z}(t)\right\rangle dt \right], \displaybreak[2]\\[2pt]
	\Delta_{3}:=\;&
	E\!\left[\int_{0}^{T}\left\langle p(t),
	-f^u_{v}(t)\hat{v}(t)-f^u_{v_{d}}(t)\hat{v}_d(t)
	\right\rangle dt\right].
\end{align*}
			Considering the initial condition of $p(t)$ and terminal conditions of $\hat{Y}(t)$, we apply Fubini's Theorem to obtain 
			$$
			\begin{aligned}
				E\left[\int_{0}^{T}\left\langle p(t), -f^u_{y_{a}}(t)\hat{Y}_{a}(t) \right\rangle dt \right]&= E\left[\int_{0}^{T}\left\langle -f^u_{y_{a}}(t)p(t), E^{\mathcal{F}_t}\left[\int_{0}^{+\infty}\phi(t,t+\theta)\hat{Y}(t+\theta)\alpha(d\theta)\right] \right\rangle dt \right]  \\
				&= E\left[\int_{0}^{T}\left\langle -f^u_{y_{a}}(t)p(t), \int_{0}^{+\infty}\phi(t,t+\theta)\hat{Y}(t+\theta)\alpha(d\theta) \right\rangle dt \right]  \\
				&= E\left[\int_{\theta}^{T+\theta}\int_{0}^{+\infty}\left\langle -f^u_{y_{a}}(t-\theta)\phi(t-\theta,t)p(t-\theta), \hat{Y}(t) \right\rangle \alpha(d\theta) dt \right]  \\
				&= E\left[\int_{0}^{T}\left\langle \int_{0}^{+\infty}-f^u_{y_{a}}(t-\theta)\phi(t-\theta,t)p(t-\theta)\alpha(d\theta), \hat{Y}(t) \right\rangle dt  \right].
			\end{aligned}
			$$	
			Since $\hat{Y}(t),\hat{Z}(t)=0, t\geqslant T$ and $l_{y_{a}}^u(t),l_{z_{a}}^u(t)=0, t\leqslant 0 $, the following duality relations also hold:
			$$
			\begin{aligned}
				E\left[\int_{0}^{T}\left\langle p(t), -f^u_{z_{a}}(t)\hat{Z}_{a}(t) \right\rangle dt \right]&= E\left[\int_{0}^{T}\left\langle \int_{0}^{+\infty}-f^u_{z_{a}}(t-\theta)\phi(t-\theta,t)p(t-\theta)\alpha(d\theta), \hat{Z}(t) \right\rangle dt  \right], \\
				E\left[\int_{0}^{T} l^u_{y_{a}}(t)\hat{Y}_{a}(t) dt \right]&= E\left[\int_{0}^{T}\left\langle \int_{0}^{+\infty} l^u_{y_{a}}(t-\theta)\phi(t-\theta,t)\alpha(d\theta), \hat{Y}(t) \right\rangle dt \right],\\
				E\left[\int_{0}^{T} l^u_{z_{a}}(t)\hat{Z}_{a}(t) dt \right]&= E\left[\int_{0}^{T}\left\langle \int_{0}^{+\infty} l^u_{z_{a}}(t-\theta)\phi(t-\theta,t)\alpha(d\theta), \hat{Z}(t) \right\rangle dt \right].\\
			\end{aligned}
			$$
			Thus, $\Delta_{1}=0$. Moreover, by the variational inequality (\ref{eq variational inequality}), we can derive that
			\begin{equation}\label{fist neccesarry condition}
				E\left[ \int_0^T\left\{-\left\langle f_v^u(t) p(t), \hat{v}(t)\right\rangle-\left\langle f_{v_d}^u(t) p(t), \hat{v}_d(t)\right\rangle +\left\langle l_v^u(t), \hat{v}(t)\right\rangle + \left\langle l_{v_d}^u(t), \hat{v}_d(t)\right\rangle \right\} d t \right]\geqslant 0.
			\end{equation}
			Note that $\hat{v}(t)=0, t\leqslant 0$ and $H^u_{v_{d}}(t)=0, t\geqslant T$, then by Fubini's Theorem, we have 
			$$
			\begin{aligned}
				E\left[\int_{0}^{T} \left\langle l_{v_{d}}^u(t)-f_{v_d}^u(t) p(t), \hat{v}_{d}(t) \right\rangle dt \right] &=E\left[\int_{0}^{T} \left\langle H^u_{v_{d}}(t), \int_{0}^{+\infty}\phi(t-\theta,t)\hat{v}(t-\theta)\alpha(d\theta)\right\rangle dt \right]\\
				&=E\left[\int_{0}^{T} \left\langle E^{\mathcal{F}_t}\left[\int_{0}^{+\infty}\phi(t,t+\theta)H^u_{v_{d}}(t+\theta)\alpha(d\theta)\right], \hat{v}(t)\right\rangle dt \right].
			\end{aligned}
			$$
			Thus, (\ref{fist neccesarry condition}) becomes 
			\begin{equation} \label{integral neccesarry condition}
				\begin{aligned}
					E\left[\int_{0}^{T}\left\langle H^u_v(t) + E^{\mathcal{F}_t}\left[ \int_{0}^{+\infty}\phi(t,t+\theta)H^u_{v_{d}}(t+\theta)\alpha(d\theta)\right],\hat{v}(t)\right\rangle dt\right] \geqslant 0.
				\end{aligned}
			\end{equation}
			Finally, we transform (\ref{integral neccesarry condition}), which is the stochastic maximum principle in integral form, into (\ref{condition neccesary SMP}). The detailed proof is similar to the proof of Theorem 1.5 in Cadenillas and Karatzas \cite{theorem15}, we omit it. Thus we complete the proof.
		\end{proof}
		
		\begin{remark}
			Condition (\ref{condition neccesary SMP}) is equivalent to the following statement:
			\begin{equation}\label{SMP without state constrain}
				\begin{aligned}
					&\left\langle H^u_v(t) + E^{\mathcal{F}_t}\left[ \int_{0}^{+\infty}\phi(t,t+\theta)H^u_{v_{d}}(t+\theta)\alpha(d\theta)\right],u(t)\right\rangle  \\
					&	= \min_{v\in U} \left\langle H^u_v(t) + E^{\mathcal{F}_t}\left[ \int_{0}^{+\infty}\phi(t,t+\theta)H^u_{v_{d}}(t+\theta)\alpha(d\theta)\right],v\right\rangle , \quad \text { a.e., a.s. }
				\end{aligned}
			\end{equation}
			This is the usual form of stochastic maximum principle without state constrains.
		\end{remark}
		
		\subsection{Sufficient Conditions.}\label{chapter sufficient SMP}
		In this section, we impose additional assumptions to obtain the sufficient conditions for the control problem (\ref{cost functional}). Let us introduce some convexity conditions:
		
		\textbf{(H3.3)} For all $t\in[0,T]$ and any given $p(t)$, $H\left(y,y_a,z,z_a, v,v_d,p(t)\right)$ is a concave function of $y,y_a,z, z_a, v,v_d$ and $\gamma$ is concave in $y$.
		
		\begin{theorem}\label{sufficient SMP}
			(Sufficient conditions of optimality) Suppose $u(\cdot)\in\mathcal{U}_{ad}$ and let $Y^u(\cdot),Z^u(\cdot)$ be the corresponding trajectories and $p(t)$ be the solution of adjoint equation (\ref{adjoint equation}). If (H3.1)-(H3.3) and (\ref{condition neccesary SMP}) (or (\ref{SMP without state constrain})) hold for $u(\cdot)$, then $u(\cdot)$ is an optimal control of the optimal stochastic control problem with delay (\ref{cost functional}) described by anticipated BSDE (\ref{ABSDE with control}).
		\end{theorem}
		\begin{proof} Take an arbitrary $v(\cdot)\in\mathcal{U}_{ad}$, let $Y^v(t),Z^v(t)$ be the corresponding trajectory. For simplicity, we denote
			$$
			\zeta^v(t)=\left(Y^v(t),Y^v_{a}(t),Z^v(t),Z^v_{a}(t)\right),\quad \zeta^u(t)=\left(Y^u(t),Y^u_{a}(t),Z^u(t),Z^u_{a}(t)\right).
			$$
			We need to prove
			$$
			J\left(v(\cdot)\right) - J\left(u(\cdot)\right)= \nabla_1 + \nabla_2 \geqslant 0,
			$$
			where 
			$$
			\begin{aligned} 	\nabla_1&:=E\left[\int_{0}^{T}\left\{l\left(t,\zeta^v(t),v(t),v_d(t)\right) - l\left(t,\zeta^u(t),u(t),u_d(t)\right)\right\}dt\right],\\  
				\nabla_2&:=E\left[\gamma\left(Y^v(0)\right) - \gamma\left(Y^u(0)\right)\right].
			\end{aligned}
			$$
			First we consider $\nabla_1$.
			$$
			\begin{aligned}
				\nabla_1	&=E\left[\int_{0}^{T}\left\{H\left(t,\zeta^v(t),v(t),v_d(t),p(t)\right) - H\left(t,\zeta^u(t),u(t),u_d(t),p(t)\right)\right\}dt\right]\\
				&\quad+E\left[\int_{0}^{T}\left\langle p(t), f\left(t,\zeta^v(t),v(t),v_d(t)\right) - f\left(t,\zeta^u(t),u(t),u_d(t)\right)\right\rangle dt\right].\\
			\end{aligned}
			$$
			Since $\left(\zeta, v, v_d\right) \rightarrow H\left(t, \zeta, v, v_d, p(t)\right)$ is concave, we have
	\begin{align*}
	\nabla_1 \geqslant\;&
	E\!\left[\int_{0}^{T}\!\Big\{
	\langle H^u_y(t), Y^v(t)-Y^u(t)\rangle
	+\langle H^u_{y_a}(t), Y_a^v(t)-Y_a^u(t)\rangle
	\Big\}dt \right]  \nonumber\\
	&+E\!\left[\int_{0}^{T}\!\Big\{
	\langle H^u_z(t), Z^v(t)-Z^u(t)\rangle
	+\langle H^u_{z_a}(t), Z_a^v(t)-Z_a^u(t)\rangle
	\Big\}dt \right] \nonumber\\
	&+E\!\left[\int_{0}^{T}\!\Big\{
	\langle H^u_v(t), v(t)-u(t)\rangle
	+\langle H^u_{v_d}(t), v_d(t)-u_d(t)\rangle
	\Big\}dt \right] \nonumber\\
	&+E\!\left[\int_{0}^{T}\!\left\langle p(t),
	f\!\left(t,\zeta^v(t),v(t),v_d(t)\right)
	-f\!\left(t,\zeta^u(t),u(t),u_d(t)\right)
	\right\rangle dt\right] \nonumber\\
	=\;&
	E\!\left[\int_{0}^{T}\!\left\langle
	H^u_y(t)+\int_{0}^{+\infty}\!\phi(t-\theta,t)H^u_{y_a}(t-\theta)\alpha(d\theta),
	Y^v(t)-Y^u(t)\right\rangle dt\right] \nonumber\\
	&+E\!\left[\int_{0}^{T}\!\left\langle
	H^u_z(t)+\int_{0}^{+\infty}\!\phi(t-\theta,t)H^u_{z_a}(t-\theta)\alpha(d\theta),
	Z^v(t)-Z^u(t)\right\rangle dt\right] \displaybreak[2]\\
	&+E\!\left[\int_{0}^{T}\!\left\langle
	H^u_v(t)+E^{\mathcal{F}_t}\!\left[\int_{0}^{+\infty}\!\phi(t,t+\theta)H^u_{v_d}(t+\theta)\alpha(d\theta)\right],
	v(t)-u(t)\right\rangle dt\right] \nonumber\\
	&+E\!\left[\int_{0}^{T}\!\left\langle p(t),
	f\!\left(t,\zeta^v(t),v(t),v_d(t)\right)
	-f\!\left(t,\zeta^u(t),u(t),u_d(t)\right)
	\right\rangle dt\right]. \nonumber
\end{align*}
			We next consider $\nabla_2$. By the concavity of $\gamma$ and applying Itô's formula to $\left\langle p(t), Y^v(t)-Y^u(t)\right\rangle$ on $[0,T]$, we have 
			$$
			\begin{aligned}
				\nabla_2 \geqslant & E\left[\gamma_y\left(Y^u(0)\right)\cdot\left(Y^v(0)-Y^u(0)\right)\right] \\
				=&-E\left[\int_{0}^{T} \left\langle H^u_y\left(t\right) + \int_{0}^{+\infty}\phi(t-\theta,t)H^u_{y_{a}}\left(t-\theta\right)\alpha(d\theta), Y^v(t)-Y^u(t)\right\rangle  dt \right]  \\
				&-E\left[\int_{0}^{T} \left\langle H^u_z\left(t\right) + \int_{0}^{+\infty}\phi(t-\theta,t)H^u_{z_{a}}\left(t-\theta\right)\alpha(d\theta), Z^v(t)-Z^u(t)\right\rangle  dt \right] \\  
				&-E\left[\int_{0}^{T}\left\langle p(t), f\left(t,\zeta^v(t),v(t),v_d(t)\right) - f\left(t,\zeta^u(t),u(t),u_d(t)\right)\right\rangle dt\right].
			\end{aligned}
			$$
			Finally, maximum condition (\ref{condition neccesary SMP}) or (\ref{SMP without state constrain}) implies 
			$$
			\begin{aligned}
				E\left[\int_{0}^{T}\left\langle H^u_v(t) + E^{\mathcal{F}_t}\left[ \int_{0}^{+\infty}\phi(t,t+\theta)H^u_{v_{d}}(t+\theta)\alpha(d\theta)\right],v(t)-u(t)\right\rangle dt\right] \geqslant 0
			\end{aligned}
			$$
			Combining all the estimates above, we can obtain 
			$$
			J\left(v(\cdot)\right) - J\left(u(\cdot)\right)= \nabla_1 + \nabla_2 \geqslant 0.
			$$
			Since $v(\cdot)\in \mathcal{U}_{ad}$ is arbitrary, $u(\cdot)$ is an optimal control.
		\end{proof}
		
		\begin{corollary}\label{corollary}
			Let us consider the case when the measure $\alpha$ is a Dirac measure at $\delta$. This corresponds to the situation where the anticipation and the delay is point-wise, i.e. $y_a(t) = y(t + \delta)$,  $z_a(t) = z(t + \delta)$, $v_d(t) = v(t - \delta)$. In this case, the system (\ref{state equation}) reduces to involve only the point-wise delay, and the sufficient condition of optimality becomes 
			$$
			\begin{aligned}
				H_v(t, \zeta^u(t),u(t),u(t-\delta),p(t)) + E^{\mathcal{F}_t}\left[ H_{v_{d}}(t+\delta, \zeta^u(t+\delta),u(t+\delta),u(t),p(t+\delta))\right]= 0,
			\end{aligned}
			$$
			where $\zeta^u(t)=\left(Y^u(t), E^{\mathcal{F}_t}\left[Y^u(t+\delta)\right],Z^u(t), E^{\mathcal{F}_t}\left[Z^u(t+\delta)\right]\right)$. The adjoint equation is given by
			\begin{equation}\nonumber
				\left\{\begin{aligned}
					dp(t) =& \left\{ -H_y\bigl(t, \zeta^u(t),u(t),u(t-\delta),p(t)\bigr) \right. \\
					&\left. -H_{y_a}\bigl(t-\delta, \zeta^u(t-\delta),u(t-\delta),u(t-2\delta),p(t-\delta)\bigr) \right\} dt \\
					&+ \left\{ -H_z\bigl(t, \zeta^u(t),u(t),u(t-\delta),p(t)\bigr)
					\right. \\
					&\left. -H_{z_a}\bigl(t-\delta, \zeta^u(t-\delta),u(t-\delta),u(t-2\delta),p(t-\delta)\bigr) \right\} dW(t), \quad t\in[0,T]; \\
					p(0) =& -\gamma_y\bigl(Y^u(0)\bigr), \quad p(t)=0, \quad t\in[-\delta,0).
				\end{aligned}\right.
			\end{equation}
		\end{corollary}

		\section{Applications} \label{chapter applications}
		
		In this section, we present three examples to illustrate the
		applicability of the stochastic maximum principle developed in Section \ref{section 3}.
		We begin with a climate policy model involving infinite anticipation and
		infinitely delayed control of exponential weighted type in Section \ref{section example carbon}, highlighting the economic relevance of long-term memory and forward-looking mechanisms.
		We then study a dynamic consumption problem with uniformly-weighted anticipation in Section \ref{section example 1}. 
		In Section \ref{section example 2}, we deal with a Linear Quadratic (LQ)  control problem with point-wise anticipated states and delayed controls.
		In all cases, the optimal controls are derived explicitly.

		
		%
		\subsection{A Climate Policy Model.}\label{section example carbon}
		
		We consider a climate policy optimization problem in which a social planner designs an environmental regulation, such as a carbon tax, in order to mitigate the long-run social costs of carbon emissions. This example comes from \cite{carbon} and the references therein.

		Let $u(\cdot)$ denote the climate policy, and let $Y(\cdot)$ represent a social cost index aggregating the expected future damages of climate change. 
		Set the anticipative term
		$$
		Y_a(t)
		=E^{\mathcal{F}_t}\!\left[\int_0^{+\infty} e^{-\lambda\theta}Y(t+\theta)\,d\theta\right],
		$$
		with the decay rate $\lambda>0$. 
		Similarly, it is reasonable to assume that the climate policy has delayed and persistent effects, we thus introduce the exponentially weighted delayed control
		$$
		u_d(t)
		=E^{\mathcal{F}_t}\!\left[\int_0^{+\infty} e^{-\mu\theta}u(t-\theta)\,d\theta\right],
		$$
		with $\mu>0$ describing the speed at which past policy loses its effectiveness.
		The admissible control set is $\mathcal{U}_{a d} \subset M_{\mathcal{F}}^2(0, T ; \mathbb{R})$, with a prescribed history of past policy $u(t)=u_0(t)$ for $t \leqslant 0$.
		
		The dynamics of social cost index $Y(\cdot)$ are described by the following infinitely anticipated BSDE:
		\begin{equation}\label{eq:climate-IABSDE}
			\left\{
			\begin{aligned}
				-dY(t) &= \big(\kappa + \beta Y_a(t) - \eta u_d(t)\big)\,dt - Z(t)\,dW(t), 
				\qquad t\in[0,T],\\
				Y(t) &= \bar Y,\quad Z(t)=0,\qquad t\in[T,+\infty).
			\end{aligned}
			\right.
		\end{equation}
		
		Economically, $\kappa>0$ represents the baseline growth rate of social cost, $\eta>0$ measures the effectiveness of climate policy, and $\beta>0$ captures the feedback effect of anticipated future damages: when the expected long-run damages are large, current social cost increases more rapidly.
		And the constant $\bar Y>0$ denotes a reference level of social cost in the far future.
		
		The planner aims to balance the policy implementation costs against the overall social cost evaluated at the initial time. 
		The objective is to minimize the performance functional
		\begin{equation}\label{eq:climate-cost}
			J(u(\cdot))
			=E\left[\int_0^T \frac{R(t)}{\theta}e^{\theta u(t)}\,dt
			+\gamma\big(Y(0)\big)\right],
		\end{equation}
		where $R(t)>0$ is a deterministic weighting function, $\theta>0$ controls the curvature of policy costs, and $\gamma(\cdot)$ is a continuously differentiable function. 
		The running cost of CARA-form reflects the rapidly increasing marginal cost of aggressive policy interventions.

		The Hamiltonian function associated with \eqref{eq:climate-IABSDE}--\eqref{eq:climate-cost} is given by
		$$
		H(t,Y_a,u,u_d,p)
		=\frac{R(t)}{\theta}e^{\theta u(t)}
		-p(t)\big(\kappa+\beta Y_a(t)-\eta u_d(t)\big).
		$$
		
		The adjoint process $p(\cdot)$ satisfies the following infinitely delayed SDE:
		\begin{equation}\label{eq:climate-adjoint}
			\left\{
			\begin{aligned}
				dp(t)
				&=\left[\int_0^{+\infty}\beta e^{-\lambda\theta}p(t-\theta)\,d\theta\right]dt,
				\qquad t\in[0,T],\\
				p(0)&=-1,\qquad
				p(t)=0,\ t<0.
			\end{aligned}
			\right.
		\end{equation}
		
		Under the assumptions of Theorem \ref{well-posedness ISDDE}, equation \eqref{eq:climate-adjoint} admits a unique solution in
		$\mathcal{M}_{\mathcal{F}}^2(-\infty,T;\mathbb{R})$. Then we have the following result for the explicit solution of optimal control.
		
		\begin{proposition}
			Let $p(\cdot)$ be the solution of the adjoint equation \eqref{eq:climate-adjoint}. Define  
			$$
			\Pi(t)
			=E^{\mathcal{F}_t}\!\left[\int_0^{+\infty}
			e^{-\mu\theta}p(t+\theta)\,d\theta\right].
			$$
			Assume that $\Pi(t)<0$ a.s. for all $t\in[0,T]$. 
			Then the optimal climate policy for problem
			\eqref{eq:climate-IABSDE}--\eqref{eq:climate-cost} is given by
			\begin{equation}\label{eq:optimal-climate-control}
				u^*(t)
				=\frac{1}{\theta}\ln\!\left(-\frac{\eta}{R(t)}\Pi(t)\right),
				\qquad t\in[0,T].
			\end{equation}
		\end{proposition}
		\begin{proof}
			By Theorem \ref{sufficient SMP}, we know that the optimal control satisfies the maximum condition
			$$
			H_u(t)
			+E^{\mathcal{F}_t}\!\left[\int_0^{+\infty}e^{-\mu\theta}
			H_{u_d}(t+\theta)\,d\theta\right]=0,
			\qquad \text{a.e. }t\in[0,T].
			$$
			Straightforward computations yield
			$$
			H_u(t)=R(t)e^{\theta u(t)},\qquad
			H_{u_d}(t)=\eta p(t).
			$$
			Substituting the expressions of $H_u$ and $H_{u_d}$, we obtain
			$$
			R(t)e^{\theta u(t)}
			+\eta\,E^{\mathcal{F}_t}\!\left[\int_0^{+\infty}
			e^{-\mu\theta}p(t+\theta)\,d\theta\right]=0.
			$$
			Provided that $\Pi(t)<0$ a.s., the optimal policy admits the explicit representation (\ref{eq:optimal-climate-control}).
		\end{proof}

		%
		%
		\begin{remark}
			This example illustrates the essential feature of the proposed framework:
			both anticipation and delay are modeled through exponentially weighted
			integral operators over unbounded horizons.
			Despite the infinite-dimensional structure,
			the stochastic maximum principle remains applicable
			and yields an explicit expression of the optimal policy.
		\end{remark}

		\subsection{Dynamic Consumption Optimization.}\label{section example 1}
		In this subsection, we investigate a dynamic consumption problem with recursive utility of uniformly-weighted type. The state process evolves according to the following anticipated BSDE:
		\begin{equation}\label{application 1 state}
			Y(t)=\xi+\int_t^T\left\{a c(s)+b E^{\mathcal{F}_s}\left[\int_s^T Y(r) d r\right]\right\} d s-\int_t^T Z(s) d W(s),
		\end{equation}
		where $a,b >0$ are given constants and $c(\cdot)$ denotes the instantaneous consumption rate without delay effect. Let $\mathcal{C}:=\left\{c(\cdot) \in \mathcal{M}_{\mathcal{F}}^2(0, T ; \mathbb{R})\right\}$ be the admissible control set. We adopt the CRRA utility function $U(x)=\frac{x^{1-\rho}}{1-\rho}$ with the relative risk aversion coefficient $\rho\in(0,1)$, and the objective is to minimize the recursive cost functional over $c(\cdot)\in \mathcal{C}$:
		\begin{equation}\label{application 1 cost functional}
			J(c(\cdot))=-E\left[\int_0^T U(c(t)) d t\right]+Y(0).
		\end{equation}
		
		Related examples can be found in \cite{BSDDE-12.delay-1type-1, Delong3}, where similar models are employed to describe the dynamic optimization problems for recursive utility of uniformly-weighted type.
		Note that the state dynamics (\ref{application 1 state}) can be reformulated as
		$$
		Y(t)=\xi+\int_t^T\left\{a c(s)+ bT E^{\mathcal{F}_s}\left[\int_{0}^T Y(s+\theta) I_{[0 ,T-s]}(\theta) \alpha(d \theta)\right]\right\} d s-\int_t^T Z(s) d W(s),
		$$
		where $\alpha$ is the uniform measure on $[0,T]$. In this case, the Hamiltonian function becomes
		$$
		H\left(t, Y_a, c, p\right)=-U(c(t))-\left\{a c(t)+bTE^{\mathcal{F}_s}\left[ \int_{0}^T Y(s+\theta) I_{[0 ,T-s]}(\theta) \alpha(d \theta)\right]\right\} p(t),
		$$
		where $p(t)$ satisfies the following adjoint equation:
		\begin{equation}\label{application 1 adjoint equation}
			\begin{aligned}
				d p(t)=  \left(\int_0^T bT p(t-\theta)I_{[0,t]}(\theta) \alpha(d\theta)\right) d t
				=  \left(\int_0^t b p(s)ds\right) d t,  \quad t\in[0,T],
			\end{aligned}
		\end{equation}
		with the initial condition $p(0)=-1$. The adjoint equation (\ref{application 1 adjoint equation}) is an ordinary differential equation (ODE) which admits a unique solution $p(t)=-\frac{1}{2}\left(e^{\sqrt{b} t}+e^{-\sqrt{b} t}\right)=-\cosh (\sqrt{b} t)$ for $t\in [0,T]$. Then from Theorem \ref{sufficient SMP}, we have the following result:
		\begin{proposition}
			For dynamic consumption optimization problem (\ref{application 1 state})–(\ref{application 1 cost functional}), the optimal consumption is $c^{*}(t)=\left(-ap(t)\right)^{-1/\rho}=\left(a\cosh (\sqrt{b} t)\right)^{-1/\rho}$ where $p(t)$ satisfies (\ref{application 1 adjoint equation}).
		\end{proposition}
		
		\subsection{The Linear Quadratic Delayed System.}\label{section example 2}
		This subsection is devoted to applying the maximum principle to study a 1-dimensional delayed linear quadratic (LQ) problem involving both point-wise anticipated states and delayed controls. In fact, LQ problems constitute a most classical family of control models, with broad applications in practice.
		
		We consider the case when the measure $\alpha$ is the Dirac measure at $\delta$, and the controlled anticipated BSDE is given by:
		\begin{equation} \label{application 2 state}
			\left\{ 
			\begin{aligned}
				-dY(t) = &\bigg\{A(t)Y(t) + B(t)E^{\mathcal{F}_t}\left[Y(t+\delta)\right] + C(t)Z(t) + D(t)E^{\mathcal{F}_t}\left[Z(t+\delta)\right] \\
				&+ E(t)v(t) + F(t)v(t-\delta)\bigg\}dt - \hat{Z}(t)dW(t), \quad~~~~ t \in [0, T]; \\
				Y(t) = \xi&(t), Z(t) = \eta(t), \quad t \in [T, T+\delta], \quad v(t) = \varphi(t), \quad t \in [-\delta, 0].
			\end{aligned}
			\right.
		\end{equation}
		
		The cost functional is
		\begin{equation}\label{application 2 cost functional}
			J(v(\cdot))= E\left[\frac{1}{2}\int_0^T\left\{L(t) v^2(t)+\tilde{L}(t)v^2(t-\delta)\right\} dt +Y(0)\right].
		\end{equation}
		
		Our optimal control problem is to minimize (\ref{application 2 cost functional}) subject to (\ref{application 2 state}) over the admissible control set $\mathcal{U}_{a d}:=\left\{v(\cdot) \in \mathcal{M}_{\mathcal{F}}^2(0, T ; \mathbb{R}), 0 \leqslant t \leqslant T\right\} $. Here, $L(\cdot),\tilde{L}(\cdot)> 0$, and $ A(\cdot)$,$B(\cdot)$,$C(\cdot)$,$D(\cdot)$,$E(\cdot)$,$F(\cdot)$ are all bounded deterministic function defined on $[0,T]$. In particular, $L^{-1}(\cdot),\tilde{L}^{-1}(\cdot)$ are also bounded.
		By Corollary \ref{corollary}, the Hamiltonian function takes the form
		$$
		\begin{aligned}
			H&\left(t,Y, Y_a,Z,Z_a,v,v_d,p\right)=\frac{1}{2}\left(L(t) v^2(t)+\tilde{L}(t)v^2(t-\delta) \right)\\
			&- \bigg\{A(t)Y(t) + B(t)E^{\mathcal{F}_t}\left[Y(t+\delta)\right] + C(t)Z(t) + D(t)E^{\mathcal{F}_t}\left[Z(t+\delta)\right]+ E(t)v(t) + F(t)v(t-\delta)\bigg\}p(t).
		\end{aligned}
		$$
		
		The associated adjoint equation is
		\begin{equation}\label{application 2 adjoint equation}
			\left\{\begin{aligned}
				d p(t)=&\left(A(t) p(t)+B(t-\delta) p(t-\delta)\right) d t\\
				&+\left(C(t) p(t)+ D(t-\delta)p(t-\delta)\right) d W(t), \quad t\in[0,T];\\
				p(0)=&-1, \quad\quad p(t)=0, \quad t\in[-\delta,0].
			\end{aligned}\right.
		\end{equation}
		
		By Theorem \ref{well-posedness ISDDE}, the adjoint equation (\ref{application 2 adjoint equation}) has a unique solution $p(\cdot) \in \mathcal{M}_\mathcal{F}^2\left(-\delta, T; \mathbb{R}\right)$. 
		Then according to Theorem \ref{sufficient SMP} (or Corollary \ref{corollary}), we have the following result:
		\begin{proposition}
			If $p(t)$ is the solution of the adjoint equation (\ref{application 2 adjoint equation}), then the optimal control of LQ optimization problem (\ref{application 2 state})–(\ref{application 2 cost functional}) is given by 
			$$
			u^{*}(t)=\frac{E(t)p(t)+I_{[0,T-\delta]}(t)F(t+\delta)E^{\mathcal{F}_t}\left[p(t+\delta)\right]}{L(t)+I_{[0,T-\delta]}(t)\tilde{L}(t+\delta)}.
			$$
		\end{proposition}
		
		We now derive an explicit solution of adjoint equation (\ref{application 2 adjoint equation}) by successive Itô's integrations over steps of length $\delta$. 
		For $t\in [0,\delta]$, (\ref{application 2 adjoint equation}) reduces to a linear SDE without delay:
		$$
		dp_0(t)=A(t)p_0(t)dt + C(t)p_0(t)dW(t),\quad p_0(0)=-1,
		$$
		which can be solved explicitly by
		$$
		p_0(t)=-\exp \left(\int_0^t A(s) d s + \int_0^t C(s) d W(s)-\frac{1}{2} \int_0^t C^2(s) d s\right)=-\Phi_0(t).
		$$
		
		Here, we define for simplicity
		$$
		\Phi_k(t)=\exp \left(\int_{k \delta}^t A(s) d s+\int_{k \delta}^t C(s) d W(s)-\frac{1}{2} \int_{k \delta}^t C^2(s) d s\right), \quad t\in[k\delta, (k+1)\delta].
		$$
		
		On the next interval $[\delta ,2\delta]$, equation (\ref{application 2 adjoint equation}) becomes
		$$
		\left\{\begin{aligned}
			d p_1(t)&=\left[A(t) p_1(t)+B(t-\delta) p_{0}(t-\delta)\right] d t+\left[C(t) p_1(t)+D(t-\delta) p_{0}(t-\delta)\right] d W(t), \quad t\in[\delta,2\delta]; \\
			p_1(t)&=p_0(t), \quad t=\delta,
		\end{aligned}\right.
		$$
		where $p_0(\cdot)$ denotes the solution of (\ref{application 2 adjoint equation}) on $[0,\delta]$. Solving the equation (\ref{application 2 adjoint equation}) on $[\delta ,2\delta]$, we get
		$$
		p_1(t)=\Phi_1(t)\left[p_1( \delta)+\int_{\delta}^t \frac{B(s-\delta) p_{0}(s-\delta)}{\Phi_1(s)} d s+\int_{ \delta}^t \frac{D(s-\delta) p_{0}(s-\delta)}{\Phi_1(s)} d W(s)\right],\quad t\in[\delta,2\delta].
		$$
		
		Iterating this procedure, for each interval $t\in [k\delta,(k+1)\delta]$, we have 
		$$
		d p_k(t)=\left[A(t) p_k(t)+B(t-\delta) p_{k-1}(t-\delta)\right] d t+\left[C(t) p_k(t)+D(t-\delta) p_{k-1}(t-\delta)\right] d W(t),
		$$
		with the initial condition $p_k(k\delta)=p_{k-1}(k\delta)$. Therefore, we obain for all $k \geqslant 1$,
		$$
		p_k(t)=\Phi_k(t)\left[p_k(k \delta)+\int_{k \delta}^t \frac{B(s-\delta) p_{k-1}(s-\delta)}{\Phi_k(s)} d s+\int_{k \delta}^t \frac{D(s-\delta) p_{k-1}(s-\delta)}{\Phi_k(s)} d W(s)\right].
		$$
		
		This provides an explicit expression for the solution of the adjoint equation (\ref{application 2 adjoint equation}).
		
		\section{Conclusion.}
		This paper studies optimal control problems for a class of stochastic systems
		governed by infinitely anticipated BSDEs
		with delayed controls. By formulating anticipation and delay terms through
		measure-valued integral operators, the proposed framework is able to capture
		both point-wise and long-term moving-average memory effects within a single setting.
		
		We first established the well-posedness of infinitely anticipated BSDEs in Theorem \ref{well-posedness IABSDE}. By
		the variational method and introducing an infinitely delayed SDE serving
		as adjoint equation, we presented the necessary and sufficient maximum principle in Theorems \ref{necessary SMP} and \ref{sufficient SMP}, respectively.
		The applicability of the theoretical results was demonstrated through several real-world examples, where the optimal controls admit explicit representations.

		Beyond the developments in this paper, several promising research directions merit further study.
		First, despite the potential usefulness of our model in applications, numerical methods for IABSDE-type systems remain largely undeveloped, thus designing efficient schemes poses a meaningful challenge.
		Second, the infinite-dimensional nature of both IABSDEs and ISDDEs suggests that an abstract infinite-dimensional control framework, such as Hilbert-space formulations, may provide new insights into dynamic programming, viscosity solutions of the associated HJB equations, and connections with evolution equations.
		\\
		\\
		\textbf{Acknowledgements.} The author wishes to thank Professor Shuzhen Yang for his insightful comments and constructive suggestions, which significantly improved the clarity and overall presentation of this manuscript.

		\bibliography{citation}

@article{07JMAA,
title = {The existence and uniqueness of the solution for stochastic functional differential equations with infinite delay},
journal = {Journal of Mathematical Analysis and Applications},
volume = {331},
number = {1},
pages = {516-531},
year = {2007},
issn = {0022-247X},
doi = {https://doi.org/10.1016/j.jmaa.2006.09.020},
url = {https://www.sciencedirect.com/science/article/pii/S0022247X06010122},
author = {Fengying Wei and Ke Wang}
}

@article{ theorem15,
Author = {Cadenillas, A and Karatzas, I},
Title = {THE STOCHASTIC MAXIMUM PRINCIPLE FOR LINEAR, CONVEX OPTIMAL-CONTROL WITH
   RANDOM-COEFFICIENTS},
Journal = {SIAM Journal on Control and Optimization},
Year = {1995},
Volume = {33},
Number = {2},
Pages = {590-624},
DOI = {10.1137/S0363012992240722},
ISSN = {0363-0129},
Unique-ID = {WOS:A1995QL45200013},
}

@article{FD-BA-1delay,
Author = {Huang, Jianhui and Shi, Jingtao},
Title = {MAXIMUM PRINCIPLE FOR OPTIMAL CONTROL OF FULLY COUPLED FORWARD-BACKWARD
   STOCHASTIC DIFFERENTIAL DELAYED EQUATIONS},
Journal = {ESAIM-Control Optimisation amd Calculus of Variations},
Year = {2012},
Volume = {18},
Number = {4},
Pages = {1073-1096},
DOI = {10.1051/cocv/2011204},
ISSN = {1292-8119},
EISSN = {1262-3377},
ORCID-Numbers = {Huang, Jianhui/0000-0003-3315-6549},
Unique-ID = {WOS:000313504300010},
}

@article{SDDE-12delay,
Author = {Øksendal, Bernt and Sulem, Agnes and Zhang, Tusheng},
Title = {OPTIMAL CONTROL OF STOCHASTIC DELAY EQUATIONS AND TIME-ADVANCED BACKWARD
   STOCHASTIC DIFFERENTIAL EQUATIONS},
Journal = {Advances in Applied Probability},
Year = {2011},
Volume = {43},
Number = {2},
Pages = {572-596},
DOI = {10.1239/aap/1308662493},
ISSN = {0001-8678},
EISSN = {1475-6064},
Unique-ID = {WOS:000292018100015},
}

@article{Oksendal2,
author = {Øksendal, Bernt and Sulem, Agnès},
year = {2000},
pages = {},
title = {A maximum principle for optimal control of stochastic systems with delay, with applications to finance},
journal = {Optimal Control and Partial Differential Equations}
}

@article{peng09AP,
author = {Shige Peng and Zhe Yang},
title = {{Anticipated backward stochastic differential equations}},
volume = {37},
journal = {The Annals of Probability},
number = {3},
publisher = {Institute of Mathematical Statistics},
pages = {877 -- 902},
year = {2009},
doi = {10.1214/08-AOP423},
URL = {https://doi.org/10.1214/08-AOP423}
}

@article{SDDE1wuzhenchenli,
title = {Maximum principle for the stochastic optimal control problem with delay and application},
journal = {Automatica},
volume = {46},
number = {6},
pages = {1074-1080},
year = {2010},
issn = {0005-1098},
doi = {https://doi.org/10.1016/j.automatica.2010.03.005},
url = {https://www.sciencedirect.com/science/article/pii/S000510981000124X},
author = {Li Chen and Zhen Wu}
}

@book{Mohammedbook1984,
author = {Mohammed, S. E. A.},
address = {Boston},
booktitle = {Stochastic functional differential equations},
isbn = {027308593X},
language = {eng},
lccn = {83024973},
publisher = {Pitman Advanced Pub. Program},
series = {Research notes in mathematics ; 99},
title = {Stochastic functional differential equations},
year = {1984},
}

@book{Mohammedbook1996,
author="Mohammed, S. E. A.",
title="Stochastic Differential Systems With Memory: Theory, Examples and Applications",
booktitle="Stochastic Analysis and Related Topics VI",
year="1998",
publisher="Birkh{\"a}user Boston",
address="Boston, MA",
pages="1--77",
isbn="978-1-4612-2022-0"
}

@article{Mohammedarticle2007,
author = {Mercedes Arriojas and Yaozhong Hu and Salah-Eldin Mohammed and Gyula Pap},
title = {A Delayed Black and Scholes Formula},
journal = {Stochastic Analysis and Applications},
volume = {25},
number = {2},
pages = {471--492},
year = {2007},
publisher = {Taylor \& Francis},
doi = {10.1080/07362990601139669},
URL = {https://doi.org/10.1080/07362990601139669},
eprint = {https://doi.org/10.1080/07362990601139669}
}

@article{Gozzi2017,
author = {Gozzi, Fausto and Masiero, Federica},
title = {Stochastic Optimal Control with Delay in the Control II: Verification Theorem and Optimal Feedbacks},
year = {2017},
issue_date = {2017},
publisher = {Society for Industrial and Applied Mathematics},
address = {USA},
volume = {55},
number = {5},
issn = {0363-0129},
url = {https://doi.org/10.1137/16M1073637},
doi = {10.1137/16M1073637},
journal = {SIAM Journal on Control and Optimization},
pages = {3013–3038},
numpages = {26}
}

@article{Gozzi2009,
author = {Gozzi, Fausto and Marinelli, Carlo and Savin, Sergei}, 
title = {On Controlled Linear Diffusions with Delay in a Model of Optimal Advertising under Uncertainty with Memory Effects},
year = {2009},
issue_date = {August    2009},
publisher = {Plenum Press},
address = {USA},
volume = {142},
number = {2},
issn = {0022-3239},
url = {https://doi.org/10.1007/s10957-009-9524-5},
doi = {10.1007/s10957-009-9524-5},
journal = {Journal of Optimization Theory and Applications},
pages = {291–321},
numpages = {31}
}

@article{solvablecase,
author = {Ismail Elsanosi and Bernt Øksendal and Agnès Sulem},
title = {Some Solvable Stochastic Control Problems With Delay},
journal = {Stochastics and Stochastic Reports},
volume = {71},
number = {1-2},
pages = {69--89},
year = {2000},
publisher = {Taylor \& Francis},
doi = {10.1080/17442500008834259},
URL = {https://doi.org/10.1080/17442500008834259}
}

@article{YU20122420,
title = {The stochastic maximum principle for optimal control problems of delay systems involving continuous and impulse controls},
journal = {Automatica},
volume = {48},
number = {10},
pages = {2420-2432},
year = {2012},
issn = {0005-1098},
doi = {https://doi.org/10.1016/j.automatica.2012.06.082},
url = {https://www.sciencedirect.com/science/article/pii/S0005109812003524},
author = {Zhiyong Yu}
}

@article{SDDE123mixeddelay,
author = {Zhang, Feng},
title = {Sufficient Maximum Principle for Stochastic Optimal Control Problems with General Delays},
year = {2022},
issue_date = {Feb 2022},
publisher = {Plenum Press},
address = {USA},
volume = {192},
number = {2},
issn = {0022-3239},
url = {https://doi.org/10.1007/s10957-021-01987-9},
doi = {10.1007/s10957-021-01987-9},
journal = {Journal of Optimization Theory and Applications},
pages = {678–701},
numpages = {24}
}

@article{Dahl03072020,
author = {Kristina Rognlien Dahl},
title = {Forward-backward stochastic differential equation games with delay and noisy memory},
journal = {Stochastic Analysis and Applications},
volume = {38},
number = {4},
pages = {708--729},
year = {2020},
publisher = {Taylor \& Francis},
doi = {10.1080/07362994.2020.1713810},
URL = {https://doi.org/10.1080/07362994.2020.1713810},
eprint = {https://doi.org/10.1080/07362994.2020.1713810}
}

@article{wangmeng25,
journal={Journal of Optimization Theory and Applications},
author={Meng Wang},
title={Maximum Principle for Optimal Control of Mean-Field Backward Doubly SDEs with Delay},
year={2025},
pages={1-24},
volume={205},
number={1},
doi={10.1007/s10957-025-02624-5},
url={https://ideas.repec.org/a/spr/joptap/v205y2025i1d10.1007_s10957-025-02624-5.html},
}

@article{Delong1,
 ISSN = {10505164},
 URL = {http://www.jstor.org/stable/20744101},
 author = {Łukasz Delong and Peter Imkeller},
 journal = {The Annals of Applied Probability},
 number = {4},
 pages = {1512--1536},
 publisher = {Institute of Mathematical Statistics},
 title = {BACKWARD STOCHASTIC DIFFERENTIAL EQUATIONS WITH TIME DELAYED GENERATORS—RESULTS AND COUNTEREXAMPLES},
 urldate = {2025-12-19},
 volume = {20},
 year = {2010}
}

@article{Delong2,
     author = {Delong, Łukasz},
     title = {Applications of time-delayed backward stochastic differential equations to pricing, hedging and portfolio management in insurance and finance},
     journal = {Applicationes Mathematicae},
     volume = {39},
     number={4},
     year = {2012},
     pages = {463-488},
     zbl = {1254.49024},
     language = {en},
     url = {http://dml.mathdoc.fr/item/bwmeta1.element.bwnjournal-article-doi-10_4064-am39-4-5}
}

@article{Delong3,
author = {Delong, Łukasz},
title = {BSDEs with Time-Delayed Generators of a Moving Average Type with Applications to Non-Monotone Preferences},
journal = {Stochastic Models},
volume = {28},
number = {2},
pages = {281--315},
year = {2012},
publisher = {Taylor \& Francis},
doi = {10.1080/15326349.2012.672281},
URL = {https://doi.org/10.1080/15326349.2012.672281},
eprint = {https://doi.org/10.1080/15326349.2012.672281}
}

@article{BSDDE-12.delay-1type-1,
Author = {Chen, Li and Huang, Jianhui},
Title = {Stochastic Maximum Principle for Controlled Backward Delayed System via
   Advanced Stochastic Differential Equation},
Journal = {Journal of Optimization Theory and Applications},
Year = {2015},
Volume = {167},
Number = {3},
Pages = {1112-1135},
DOI = {10.1007/s10957-013-0386-5},
ISSN = {0022-3239},
EISSN = {1573-2878},
ORCID-Numbers = {Huang, Jianhui/0000-0003-3315-6549}
}

@article{FD-BA-1delay-impulse,
  author={Wu, Zhen and Zhang, Feng},
  journal={IEEE Transactions on Automatic Control}, 
  title={Stochastic Maximum Principle for Optimal Control Problems of Forward-Backward Systems Involving Impulse Controls}, 
  year={2011},
  volume={56},
  number={6},
  pages={1401-1406},
  doi={10.1109/TAC.2011.2114990}}

@INPROCEEDINGS{BSDDE12delay,
  author={Shi, Jingtao},
  booktitle={Proceedings of the 30th Chinese Control Conference}, 
  title={Optimal control of backward stochastic differential equations with time delayed generators}, 
  year={2011},
  volume={},
  number={},
  pages={1285-1289}
}

@article{BSDDE-12delay-type1+IH,
 author={Liang, Hong and Zhou, Jianjun},
 title={INFINITE HORIZON OPTIMAL CONTROL PROBLEMS OF BACKWARD STOCHASTIC DELAY DIFFERENTIAL EQUATIONS IN HILBERT SPACES},
 journal={Bulletin of the Korean Mathematical Society},
 publisher={대한수학회},
 volume={57},
 number={2},
 pages={311-330},
 year={2020}
}

@article{ISDDEapplication1,
Author = {Hino, Y and Murakami, S and Naito, T},
Title = {FUNCTIONAL-DIFFERENTIAL EQUATIONS WITH INFINITE DELAY},
Journal = {Lecture Notes in Mathematics},
Year = {1991},
Volume = {1473},
Pages = {1-\&},
ISSN = {0075-8434},
Unique-ID = {WOS:A1991GH51800001},
}

@book{ISDDEapplication2,
Title = {Stochastic Differential Equations and Applications},
Author = {Xuerong Mao},
Publisher={Horwood},
Year = {1997},
Language = {English},
ISBN = {1898563268},
}

@article{ziji,
      title={Infinite Anticipation Backward Stochastic Differential Equations}, 
      author={Guanwei Cheng and Shuzhen Yang},
      year={2025},
      eprint={2511.15548},
      journal = {arXiv e-prints},
      pages = {arXiv:2511.15548},
      archivePrefix={arXiv},
      primaryClass={math.PR},
      url={https://arxiv.org/abs/2511.15548}, 
}

@article{carbon,
title = {Cumulative carbon emissions and economic policy: In search of general principles},
journal = {Journal of Environmental Economics and Management},
volume = {96},
pages = {108-129},
year = {2019},
issn = {0095-0696},
doi = {https://doi.org/10.1016/j.jeem.2019.04.003},
url = {https://www.sciencedirect.com/science/article/pii/S0095069618302122},
author = {Simon Dietz and Frank Venmans}
}
		
	\end{document}